\documentclass[12pt,reqno]{amsart}
\usepackage{amssymb}

\usepackage{amscd}

\newcommand{\RNum}[1]{\uppercase\expandafter{\romannumeral #1\relax}}

\usepackage{graphicx}

\usepackage[T2A]{fontenc}
\usepackage[utf8]{inputenc}
\usepackage[russian,english]{babel}
\input{int.def}

\usepackage{mdwlist}

\usepackage{hyperref}
\usepackage{blindtext}
\usepackage{scrextend}

\usepackage{tikz}
\usetikzlibrary{matrix,arrows,calc}
\usepackage{tikz-cd}
\usepackage{graphicx}
\usepackage{caption}
\usepackage{subcaption}
\usepackage{wrapfig}

\usepackage{amsmath}

\usepackage{xcolor}
\usepackage{array}

\usepackage{comment}
\usepackage{enumitem}

\numberwithin{equation}{section}

\myitemmargin 
\def\o{\omega}
\def\rp{\mathbb{RP}}

\def\der{\text{\normalfont{d}}}

\begin{document}

\myitemmargin

\title[Symplectic triangle inequality]%
{Symplectic triangle inequality}


\author{Vsevolod Shevchishin}
\address{Faculty of Mathematics and Computer Science\\
University of Warmia and Mazury\\
ul.~Słoneczna 54, 10-710 Olsztyn, Poland
}
\email{shevchishin@gmail.com}
\thanks{}

\author{Gleb Smirnov}
\address{ETH Zürich}
\curraddr{}
\email{gleb.smirnov@math.ethz.ch}
\thanks{}

\subjclass[2010]{Primary }

\date{}

\dedicatory{}

\begin{abstract}
We prove a non-squeezing result for Lagrangian embeddings of the real projective plane into blow-ups of the symplectic ball.
\end{abstract}

\maketitle



\setcounter{section}{0}

\section{Statement of the main result}
The problem we consider is whether or not one may find 
an embedded Lagrangian $\rp^2$ in the three-fold blow-up of the symplectic ball.
Let $(B,\o)$ be the symplectic ball with $\int_{B} \o^2 = 1$, and 
let $B_3(\mu_1,\mu_2,\mu_3)$ be $B$ blown-up three times; here $\mu_i > 0$ are the areas of the exceptional curves, which satisfy $1 - \mu_i - \mu_j > 0$.
Note that the positivity condition $1 - \sum_i \mu_i^2 > 0$ is 
automatically satisfied.

We will show that 
$B_3(\mu_1,\mu_2,\mu_3)$
admits an embedded Lagrangian $\rp^2$ if and only if
$\mu_i$ obey 
\[
\mu_i < \mu_j + \mu_k,
\]
so that the sum of the sizes of any two blow-ups must be greater than the 
size of the remaining blow-up. The existence of a Lagrangian $\rp^2$ in $B_3$ 
has been previously reported in \cite{BLW}, under the assumption that $\mu_i$ 
are equal to each other and sufficiently small.

Although it is immediate that there is no embedded Lagrangian $\rp^2$ in the symplectic ball $B$,
one may ask if there is one in the blow-up of $B$ or the two-fold blow-up of $B$. The answer to this question 
is negative as there is a topological obstruction to such an embedding; a result due to Audin \cite{Aud} says that if 
$L$ is an embedded Lagrangian $\rp^2$ then 
\[
[L]^2 = 1\ \text{mod}\,4.
\]
(The reader will recall here that 
the self-interestion number of mod 2 
classes has a lift to $\zz_4$ coefficients,
the Pontrjagin square.) It is easy to see that neither the blow-up of $B$ nor the two-fold blow-up has suitable homology classes.

There is no general method to find obstructions for Lagrangian embeddings into symplectic $4$-manifolds, though there are many results known. 
For instance, Li and Wu show (see \cite{LW}) there exists an embedded Lagrangian sphere in the two-fold blow-up of $B$ if and only if the sizes of the blow-ups are equal to each other.

Although one can always find an embedded Lagrangian torus in $B$, such an embedding must satisfy interesting symplectic constraints. We let $\alpha$ to denote the action form on $B$, $\text{\normalfont{d}}\,\alpha = \o$. If $T^2$ is a Lagrangian torus in $B$, then the restriction of $\alpha$ to $T^2$ is closed and, therefore, defines a class in $\sfh^1(T^2;\rr) \cong \rr^2$.
A classical result of Gromov says (see \cite{Gro}) that $[\alpha]$ never vanishes. 
In \cite{HO}, Hind and Opshtein established a certain bound on the size of $B$ in terms of $[\alpha] \in \sfh^1(T^2;\rr)$. 

It is shown by Nemirovski-Shevchishin (see \cite{N,Sh}) that there is no Lagrangian embedding of the Klein bottle into $B$.

\state Acknowledgements.  
We are indebted to Silvia Anjos and Rosa Sena-Dias for 
useful discussions. We also wish to thank Jeff Hicks for 
reading the manuscript and for helpful comments.

\section{Preliminaries}

\subsection{Symplectic rational blow-up}
For symplectic $4$-manifolds, the standard blow-down is performed by removing a neighbourhood of a symplectic sphere with 
self-inter\-section $-1$ and replacing the sphere with the standard symplectic $4$-ball. 
The \slsf{symplectic rational blow-down} involves replacing a neighbourhood of a symplectic $(-4)$-sphere 
with the symplectic rational homology ball which is the standard symplectic neighbourhood of $\rp^2$ in $T_{\rp^2}$. 
For details, see \cite{F-S, Sym-1}, where more general blow-downs are considered.

A different viewpoint comes from the symplectic sum surgery introduced in \cite{MW,Gm}. Consider two symplectic $4$-manifolds $(X_i,\o_i)$, $i=1,2$, which contain symplectic spheres $S_i$ with 
\[
[S_1]^2 = -[S_2]^2\quad \text{and}\quad \int_{S_1} \o_1 = \int_{S_2} \o_2.
\]
Let $\overline{X_i - S_i}$ be the manifold with boundary such that
$\overline{X_i - S_i} - Y_i$ is symplectomorphic to $(X_i - S_i, \o_i)$,
where $Y_i = \del(\overline{X_i - S_i})$ is diffeomorphic to a circle bundle over $S_i$.
The symplectic sum $X_1 \#_{S_1 = S_2} X_2$ is defined as $\overline{X_1 - S_1} \cup_{\varphi} \overline{X_2 - S_2}$, where $\varphi \colon Y_1 \to Y_2$ is an orientation-reversing diffeomorphism.

One may equip $X_1 \#_{S_1 = S_2} X_2$ with a symplectic structure $\o$ which agrees with $\o_i$ over $X_i - S_i$ and whose properties can be recovered from those of $\o_i$. For instance, 
\[
\int_{X_1 \#_{S_1 = S_2} X_2} \o^2 = \int_{X_1} \o_1^2 + \int_{X_2} \o_2^2.
\]
There are various descriptions of the symplectic sum available in the literature; the one in 
\cite{Sym-2} is particularly visual.

Let $(\wt{X},\o)$ be a symplectic $4$-manifold containing a symplectic $(-4)$-sphere $\Sigma$, and 
let $\o_0$ be the Fubini-Study symplectic form on $\cp^2$. One may perform 
the symplectic sum 
\begin{equation}\label{eq:split}
X := \wt{X} \#_{\Sigma = Q} \cp^2,
\end{equation}
where $Q \subset \cp^2$ is the quadric 
$Q = \left\{ z_0^2 + z_1^2 + z_2^2 = 0 \right\}$. 
Note that we need to scale $\o_0$ up such that
$$
\int_{\Sigma} \o = \int_{Q} \o_0.
$$
Note also that the complement of $Q$ in $\cp^2$ is a symplectic neighbourhood 
of the Lagrangian projective plane $\left\{ z_i = \bar{z}_i \right\}$, and the Lagrangian therefore embeds into $X$. 

Since a symplectic neighbourhood of an embedded Lagrangian $\rp^2$ is entirely standard, the rational blow-down surgery is reversible. Namely, whenever $X$ contains an embedded Lagrangian $L \cong \rp^2$,
there exists a positive sufficiently small $\varepsilon$ such that $X$ splits according to \eqref{eq:split} with $\int_{Q} \o_0 = 4\,\varepsilon$. 

We shall say that the manifold $\wt{X}$ in \eqref{eq:split} is the \slsf{symplectic rational blow-up} of $L$
in $X$. Then the value of $4\,\varepsilon$, which may be chosen arbitrary small, is called the \slsf{size of the rational blow-up}. See \cite{Kh-1,Kh-2} for a detailed study of symplectic rational blow-ups.

If $X$ is the rational blow-down of $\Sigma$ from $\wt{X}$, then 
\begin{equation}\label{eq:b}
b_1(X) = b_1(\wt{X}), \quad
b_2^{+}(X) = b_2^{+}(\wt{X}), \quad 
b_2^{-}(X) = b_2^{-}(\wt{X}) - 1.
\end{equation}
These equations follow from \cite{F-S}. We now discuss the 
relation between 
the intersection form of $X$ and that of $\wt{X}$ in detail.

\subsection{Lattice calculation.}\label{lattice} 
In this note a \slsf{lattice} is a free Abelian group $\Lambda\cong \zz^n$
equipped with a non-degenerate symmetric bilinear form $q_\Lambda: \Lambda \times \Lambda \to \zz$. 

Let $(X,\omega)$ be a  compact symplectic manifold, $L \cong \rp^2$ be a 
Lagrangian in $X$, and $(\wt{X},\ti\o)$ be the rational blow-up 
of $L$ in $X$. Denote by $\Sigma$ the resulting exceptional $(-4)$-sphere, by $\Lambda:=\sfh_2(X,\zz)/\mathrm{Tor}$ the $2$-homology lattice of $X$, 
and by $\wt\Lambda:= \sfh_2( \wt X,\zz) /\mathrm{Tor}$ the same lattice of $\wt X$.

\smallskip
Following \cite{BLW}, we describe 
the relation of $\wt\Lambda$ to $\Lambda$.
The intersection 
with $L \cong\rp^2$ defines a homomorphism 
$w_L \colon \Lambda \to \zz_2$. 
Denote by $\Lambda'$ the kernel of this homomorphism. 
This is a sublattice of $\Lambda$ of index $2$. 

The elements of $\Lambda'$ are represented by oriented surfaces in $X$ having vanishing $\zz_2$-intersection index with $L$. 
By placing the surface $Y$ in generic position we obtain an even number of transverse intersection points of $Y$ with $L$. The intersections points can easily be made to disappear, 
by cutting from $Y$ a small neighbourhood of each intersection point and 
connecting the boundaries by tubes. If desired, the surgery can be done 
in such a way that the obtained surface remains orientable, see \slsf{Lemma 4.10} 
in \cite{BLW}.

We therefore conclude that $\Lambda'$ is the $2$-homology lattice of $X\bs L$.
Since there exists a natural diffeomorphism $X\bs L  \cong \wt X \bs \Sigma$, we obtain a natural embedding $\Lambda'\subset \wt \Lambda$. The image of the latter will be denoted by $\wt\Lambda'$.

On the other hand, the homology class of $\Sigma$ generates the sublattice $\zz\lan[\Sigma]\ran \subset \wt \Lambda$ of rank $1$. 
In a similar vein as above one shows that the orthogonal sublattice $[\Sigma]^\perp$ is generated by oriented surfaces disjoint from $\Sigma$, and that sublattice is canonically identified with $\wt\Lambda'$.
If $S$ is an oriented embedded surface in $X$ such that $[S] \in [\Sigma]^\perp$, 
then one constructs a representative of $[S]$ that is disjoint from $\Sigma$ as follows.
Arrange $S$ to be transverse to $\Sigma$ so that they intersect each other in finitely many points 
$Q_1,\ldots,Q_k$.
Pick a pair of points $Q_1,Q_2$ of opposite signs; we want to get rid of them.
Let $\Gamma_1$ and $\Gamma_2$ be small circles in $S$ going around the points $Q_1$ and $Q_2$, respectively.
Pick a path $\gamma \subset \Sigma$ from $Q_1$ to $Q_2$. 
Then, using a thin tube following the chosen path, we can connect $\Gamma_1$ to $\Gamma_2$. 
The intersections $Q_1$ and $Q_2$ have now been eliminated. The number of positive points $Q_i$ must be equal to the number of negative $Q_i$, or $[\Sigma] \cdot [S]$ would not have vanished. 
So pick another pair of points, find a path between them, eliminate, and so on till we run out of intersection point.

Thus the sum $\wt\Lambda'\oplus \zz\lan[\Sigma]\ran$ is orthogonal, and this is a sublattice in $\wt\Lambda$ of finite index. 

The index of $\big[\wt \Lambda: \wt\Lambda'\oplus \zz\lan[\Sigma]\ran \big]$ is the square root of the discriminant of the lattice $\wt\Lambda'\oplus \zz\lan[\Sigma]\ran$.
Recall that the \slsf{discriminant} of a lattice is the absolute value of the Gram matrix of the lattice with respect to any basis. 
Since the sum  $\wt\Lambda'\oplus \zz\lan[\Sigma]\ran$ is orthogonal, this discriminant is the product
of the discriminants of $\wt\Lambda'$ and $\zz\lan[\Sigma]\ran$. 

The first discriminant is $4=2^2$ since $\Lambda' \cong \wt\Lambda'$ has index $2$ in the unimodular lattice $\Lambda$. In the case of $\zz\lan[\Sigma]\ran$
the discriminant is $|\Sigma^2|=|-4|=4$. It follows that discriminant of the lattice $\wt\Lambda'\oplus \zz\lan[\Sigma]\ran$ is $4\cdot4=16$, and so the index is $4$. In particular, for every $\lambda\in\wt\Lambda$ the multiple $4\,\lambda$ lies in $\Lambda'\oplus \zz\lan[\Sigma]\ran$.
 \smallskip%

We sum up our previous considerations as follows:
\begin{lem} Let $(X,\omega)$ be a closed symplectic $4$-manifold, and $L\subset X$ be Lagrangian real projective plane in $X$.
Let $(\wt X, \wt\omega)$ be the symplectic rational blow-up of $L$ in $X$, and $\Sigma$ the arising $(-4)$-sphere. Denote by $\Lambda$ and $\wt\Lambda$ the integer lattices of $X$ and resp.\ $\wt X$. Let $\Lambda'$ be the sublattice of vectors $\lambda \in \Lambda$ having vanishing $\zz_2$-intersection with $L$.

Then the lattice $\wt \Lambda$ admits a sublattice naturally isomorphic to $\Lambda' \oplus \zz\lan[\Sigma]\ran$, and the quotient group is $\zz_4$. (This follows from unimodularity of $\wt\Lambda$.)

Since the rational blow-up surgery does not affect the symplectic form $\o$ away from some tubular neighbourhood of $L$, we see that the Chern class $c_1(\wt X)$ coincides with the class $c_1(X)$ on the sublattice $\Lambda'$, and so do the classes $[\o]$ and $[\wt{\o}]$.
\end{lem}

\section{The inequalities}

We define a \slsf{symplectic ball} $B_0$ as the round ball of radius $r$ in $\rr^4=\cc^2$ equipped with the standard symplectic structure
\[
\o_{0}:= \frac{\cpi}{2} \big( \der z_1 \wedge \der\bar{z}_1 +
\der z_2 \wedge \der\bar{z}_2 \big)
= \der x_1 \wedge \der y_1 + \der x_2 \wedge \der y_2.
\]
In this case we say that the quantity $\pi r^2$ is the \slsf{size of the ball} $B_0$. This is the $\omega_0$-area of the disc $\{ (x_1,y_1;0,0)\;:\; x_1^2 +y_1^2 \le1\}$ in $B$.  

Take the symplectic ball $(B_0,\o_{0})$ of size $1$. Inside $B_0$ take three disjoint symplectic balls $B(x_i,\mu_i)$, $i = 1,2,3$, of sizes $\mu_i > 0$ and centers $x_i$. By $B_3(\mu_1, \mu_2, \mu_3)$ we denote the three-fold blow-up of $B_0$ at $x_i$, and by $E_i \subset B_3(\mu_1, \mu_2, \mu_3)$ we denote the arising exceptional spheres.

\subsection{Construction of Lagrangian $\rp^2$’s in a
triply blown-up ball.} For this discussion we follow closely \slsf{§ 4.3.1} in \cite{BLW}. 

Take the symplectic ball $(B_0, \o_{0})$ of size $1$. 
Inside $B_0$ take a symplectic ball $B(\wt{x}_0,\wt\mu_0)$ of size 
$\wt\mu_0 > 0$ and center $\wt{x}_0$. 
Let $(B_1,\wt\omega_1)$ be the symplectic blow-up of the ball $(B_0,\omega_{0})$ at $\wt{x}_0$ of size $\wt\mu_0$, using the ball $B(\wt{x}_0,\wt\mu_0)$. 
Denote by $\wt{E}_0$ the arising exceptional sphere. 
Then 
$\int_{\wt{E}_0}\wt\omega_1=\wt\mu_0$.

Take three distinct points $\wt{x}_1, \wt{x}_2, \wt{x}_3$ on $\wt{E}_0$. 
Then there exist disjoint symplectic balls $B(\wt{x}_i,\wt\mu_i)$ of some sizes $\wt{\mu}_i > 0$ such that each intersection 
$\wt{E}_0 \cap B(\wt{x}_i,\wt{\mu}_i)$ is a disc $D(\wt{x}_i,\wt{\mu}_i)$ of area $\wt{\mu}_i$. 
Notice that we get 
$\wt{\mu}_1 + \wt{\mu}_2 + \wt{\mu}_3 < \wt{\mu}_0$. 

Let $(B_4,\wt{\omega}_4)$ be the three-fold symplectic blow-up of the domain $(B_1, \wt{\omega}_1)$ at the points $\wt{x}_i$ using the balls $B(\wt{x}_i,\wt{\mu}_i)$. 
Denote by $\wt{E}_1, \wt{E}_2, \wt{E}_3$ the arising exceptional spheres. 
Then $\int_{\wt{E}_i}\wt{\omega}_4 = \wt{\mu}_i$. 
The proper preimage of $\wt{E}_0$ in $(B_4,\wt{\omega}_4)$ is a symplectic sphere $\Sigma$ of homology class $[\Sigma] = 
[\wt{E}_0] - ([\wt{E}_1] + [\wt{E}_2] +[\wt{E}_3])$ and of area 
$\int_{\Sigma} \wt{\omega}_4 
= \wt{\mu}_0 - (\wt{\mu}_1 + \wt{\mu}_2 + \wt{\mu}_3)$.

Recall that there exists a symplectic embedding $(B_0,\omega_0) \subset (\cp^2,\omega_{st})$ such that the complement of $B_0$ in $\cp^2$ is a projective line $H$. 
Here $\o_{st}$ is the Fubini-Study form on $\cp^2$ normalized by $\int_{\cp^2} \o_{st}^2 = 1$.
A classical result of Lalonde-McDuff \cite{LaMc} says that if a symplectic $4$-manifold $X$ contains an embedded symplectic sphere of non-negative self-intersection number, then $X$ is either rational or ruled (not necessarily minimal.) If, moreover, there is an embedded sphere of positive self-intersection number, then $X$ is either $S^2 \times S^2$ or is $\cp^2$ blown-up a number of times.
This implies that every symplectic domain for which a collar neighbourhood 
of its boundary is symplectomorphic to that of $B_0$ is obtained from $B_0$ by finite sequence of symplectic blow-ups. 
Consequently, the rational blow-up of a Lagrangian projective plane in 
$B_3(\mu_1,\mu_2,\mu_2)$ is $(B_4,\wt{\omega}_4)$ (as the rational blowing-up surgery is performed away from $\del B_3$.)

\subsubsection{Necessity.}
Let us make the homology lattice comparison of $B_3$ and $B_4$.
For this purpose we use an embedding $B_0$ in $\cp^2$ for which 
$\cp^2 = B \sqcup H$, where $H \subset \cp^2$ is a projective line.
We use the notation $X_3,\wt{X}_4$ for the $\cp^2$ blown-up $3$ or resp.\ $4$ times. 
We obtain the lattices
\[
\Lambda_3 := \sfh_2(X_3,\zz) 
= \zz \lan\; [H],\, [E_1],\,[E_2],\,[E_3] \;\ran,
\]
\[
\Lambda_4 := \sfh_2(\wt{X}_4,\zz) 
= \zz \lan\; [H],\, [\wt{E}_0],\, [\wt{E}_1],\,[\wt{E}_2],\,
[\wt{E}_3]\; \ran,
\]
where $[H]$ denotes the class of the line in $\cp^2$. 
In this notation we have 
\[
[L]_{\zz_2} \equiv [E_1] +[E_2]+[E_3] \mod 2
\]
in $X_3$, and 
\begin{equation} \label{Sigma=E}
[\Sigma]= [\wt{E}_0]-(\, [\wt{E}_1]+[\wt{E}_2]+ [\wt{E}_3]\,)
\end{equation}
in $\wt{X}_4$. 
The latter follows from the equations
\[
[\Sigma] \cdot [H] = 0,\quad [\Sigma]^2 = -4,\quad c_1(\wt{X}_4) \cdot [\Sigma] = -2.
\]
Indeed, the orthogonality condition $[\Sigma] \cdot [H] = 0$ implies that 
$[\Sigma]$ can be written in the form
\[
[\Sigma] = l_0[\wt{E}_0] + l_1[\wt{E}_1] + l_2[\wt{E}_2] + l_3[\wt{E}_3].
\]
Since $[\Sigma]^2 = -4$, it follows that $l^2_i = \pm 1$. 
But only one of $l_i$ can be positive, or $c_1(\wt{X}_4) \cdot [\Sigma]$ would not 
be equal to $(-2)$. 
We conclude that $[\Sigma]$ is unique up to permutation of $[\wt{E}_i],\ i = 0,\ldots,3$.

Further, the Chern classes of $X_3$ and $\wt{X}_4$ are
\[
c_1(X_3) = 3[H] - (\, [E_1]+[E_2]+ [E_3]\,),
\qquad
c_1(\wt{X}_4) = 3[H] - (\,[\wt{E}_0]+ [\wt{E}_1]+[\wt{E}_2]+ [\wt{E}_3]\,).
\]
Next, recall that we have the sublattice $\Lambda'_3$ consisting of vectors $\lambda\in \Lambda_3$ such that $\lambda\cdot [L]\equiv 0 \mod 2$. The sublattice $\Lambda'_3$ is generated by $[H]$ and the classes $[E_i]-[E_j]$, $2[E_i]$. The latter are primitive in $\wt{\Lambda}_4$, orthogonal to $[H]$, and characterised by the properties 
\[
([E_i]-[E_j])^2=-2,\quad c_1\cdot ([E_i]-[E_j])=0, \quad
(2[E_i])^2=-4,\quad c_1\cdot (2[E_i])=2.
\]
Let us consider the sublattice $\wt{\Lambda}'_4 \subset \wt{\Lambda}_4$ consisting of the vectors $\lambda\in \Lambda_4$ orthogonal to $[\Sigma]$ and find the classes with the properties above in $\wt{\Lambda}'_4$. The orthogonality to $[H]$ means that we seek vectors of the form 
\begin{equation} \label{la=sumEi}
    \lambda= k_0[\wt{E}_0] + k_1[\wt{E}_1] +k_2[\wt{E}_2] + k_3[\wt{E}_3]. 
\end{equation}

The condition $\lambda^2=-2$ means that two of the coefficients $k_0,\ldots,k_3$ are $0$ and two of them $\pm1$. The orthogonality to $[\Sigma]$ leaves two possibilities:
either $[\wt{E}_i]-[\wt{E}_j]$ with $i\ne j \in \{1,2,3\}$ or $\pm(\,[\wt{E}_0] + [\wt{E}_i]\,)$ with $i=1,2,3$. 
The orthogonality to $c_1$ excludes the latter possibility. The classes with $\lambda^2=-4$ are either $2[\wt{E}_0],2[\wt{E}_i]$, or with coefficients $k_i=\pm1$ in \eqref{la=sumEi}. The orthogonality to $[\Sigma]$ excludes double classes $2[\wt{E}_0],2[\wt{E}_i]$ and says that two of the coefficients $k_0,\ldots,k_3$ are the same as for $[\Sigma]$ and two of the opposite sign. Finally, the condition $c_1\cdot \lambda=2$ says that one of the coefficients $k_0,\ldots,k_3$ is $-1$ and three other are $+1$. So our classes $\lambda$ with $\lambda^2=-4$ are
\[
[\wt{E}_0]+ [\wt{E}_1]+ [\wt{E}_2] -[\wt{E}_3],\quad
[\wt{E}_0]+ [\wt{E}_1]- [\wt{E}_2] +[\wt{E}_3],\quad
[\wt{E}_0] -[\wt{E}_1]+ [\wt{E}_2] +[\wt{E}_3].
\]
Notice that the symmetric group $\sym_3$ permuting the classes in the sets  $\{ [E_1], [E_2], [E_3] \}$ and $\{ [\wt{E}_1], [\wt{E}_2], [\wt{E}_3] \}$
acts in compatible way on the generating classes of the lattices $\Lambda'_3$ and $\wt{\Lambda}'_4$. 

The last property we need is 
\[
\begin{split}
& [E_i] -[E_j] = \tfrac{1}{2}\,(\,2[E_i] -2[E_j]\,) \qquad
\text{in } \Lambda'_3,
\\
& [\wt{E}_1] -[\wt{E}_2] = \tfrac{1}{2}\,\big(\,
([\wt{E}_0]+ [\wt{E}_1]- [\wt{E}_2] +[\wt{E}_3])
- ([\wt{E}_0]- [\wt{E}_1] +[\wt{E}_2] +[\wt{E}_3])\,\big) \qquad
\text{in } \wt{\Lambda}'_4
\end{split}
\]
and similar for $[\wt{E}_1] -[\wt{E}_3]$, $[\wt{E}_2] -[\wt{E}_3]$.

Summing up we conclude:
\begin{lem} There is a unique (up to $\sym_3$) lattice isomorphism $\Lambda'_3 \to \wt{\Lambda}'_4$ which sends
\begin{equation}  \label{E-corresp}
    \begin{split}
& 2[E_1] \mapsto [\wt{E}_0] -[\wt{E}_1] +[\wt{E}_2] +[\wt{E}_3], \quad
   2[E_2] \mapsto [\wt{E}_0] +[\wt{E}_1] -[\wt{E}_2] +[\wt{E}_3], \\[2pt]
& 2[E_3] \mapsto [\wt{E}_0] +[\wt{E}_1] +[\wt{E}_2] -[\wt{E}_3]. 
\end{split}
\end{equation}
On the other hand, those lattice isomorphisms which preserve $c_1$ satisfy \eqref{E-corresp}.
\end{lem}

Now we can give a proof of the triangle inequality. Let $(B_3, \omega_3)$ be a symplectic ball blown-up triply, and $E_1,E_2,E_3$ the corresponding exceptional spheres. Denote by $\mu_i:= \int_{E_i} \omega_3$ the periods of the symplectic form so $(B_3, \omega_3)$ is $B_3(\mu_1,\mu_2,\mu_3)$.

Assume that there exists a Lagrangian $L \cong \rp^2$ in $(B_3,\omega_3)$. Let $(B_4,\wt{\omega}_4)$ be the symplectic rational blow-up of $L$ of size $\varepsilon>0$. Introduce the homology classes in $\sfh_2(B_4,\zz)$ according to the formulas \eqref{Sigma=E} and \eqref{E-corresp}. Set $\wt{\mu}_i := \int_{\wt{E}_i} \wt\omega_4, i = 0,\ldots,3$.
We have the relations:
\[
\begin{split}
& \wt{\mu}_0 - (\wt{\mu}_1 + \wt{\mu}_2 +\wt{\mu}_3) = 4\varepsilon \\
& \wt{\mu}_0 - \wt{\mu}_1 + \wt{\mu}_2 +\wt{\mu}_3 = 2\mu_1 \qquad
\wt{\mu}_0 +\wt{\mu}_1 - \wt{\mu}_2 +\wt{\mu}_3 = 2\mu_2 \qquad
\wt{\mu}_0 +\wt{\mu}_1 + \wt{\mu}_2 -\wt{\mu}_3 = 2\mu_3 
\end{split}
\]
or resolved in $\wt{\mu}_i$
\begin{equation} \label{mu2mu'}
\begin{split}
& \wt{\mu}_0 = \frac{\mu_1 + \mu_2 +\mu_3}{2}  + \varepsilon 
\\[2pt]
& \wt{\mu}_1  = \frac{\mu_2 + \mu_3 -\mu_1}{2} - \varepsilon 
\qquad
\wt{\mu}_2 = \frac{\mu_1 + \mu_3 -\mu_2}{2} -\varepsilon 
\qquad
\wt{\mu}_3 = \frac{\mu_1 + \mu_2 -\mu_3}{2} -\varepsilon.
\end{split}
\end{equation}
The latter formulas not only demonstrate the symplectic triangle inequality, but also give the upper bound on the maximal possible size of the rational symplectic blow-up.

\subsubsection{Sufficiency.}
We let $\omega_3$ to denote the symplectic form on $B_3(\mu_1,\mu_2,\mu_3)$. 
Let us extend $\omega_3$ to a symplectic form on $X_3$, the three-fold blow-up of $\cp^2$. 
We use the same notation $\omega_3$ for the extension; we get 
\[\textstyle
[\omega_3] = [H] - \sum_i \mu_i [E_i].
\]

We assume $\omega_3$ to satisfy:
\begin{enumerate}
\item \label{NM1}
$[\omega_3]^2 = 1 - \sum_i \mu_i^2 >0$ (``positive volume'');
\item \label{NM2}
$\mu_i>0$ and $\mu_i+\mu_j<1$ (``effectivity of exceptional curves''); 
\item \label{NM.RP2}
$\mu_i+\mu_j> \mu_k$, the latter is the symplectic triangle inequality.  
\end{enumerate}
Let us show that under the additional condition \eqref{NM.RP2} there exist a Lagrangian $L\cong \rp^2$ in $(X_3,\omega_3)$ disjoint from the line $H$. For this purpose we fix some sufficiently small $\varepsilon>0$ and define new periods $\wt{\mu}_0,\ldots,\wt{\mu}_3$ by \eqref{mu2mu'} so that they are positive and satisfy
\begin{equation}\label{mu'ineq}
1 - \sum_{i = 0}^{3} \wt\mu_i^2 > 0,
\quad
\wt\mu_0 - (\wt\mu_1 + \wt\mu_2 + \wt\mu_3) > 0,
\quad
1 - \wt\mu_0 - \wt\mu_i > 0,
\ i = 1,2,3.
\end{equation}
Now, consider a line $H$ in $\cp^2$ and a point $\wt{x}_0 \in \cp^2$ that does not lie on $H$. 
Let $\wt{X}_1$ be the blow-up of $\cp^2$ at $\wt{x}_0$, and let $\wt{E}_0$ be the arising exceptional curve.  
After that, take three distinct points $\wt{x}_1,\wt{x}_2,\wt{x}_3$ on $\wt{E}_0$ and blow-up $\wt{X}_1$ at them. 
Denote by $\wt{X}_4$ the resulting complex surface and by $\wt{E}_i, i =1,2,3$ the corresponding exceptional complex curves. The proper preimage of $\wt{E}_0$ in $\wt{X}_4$, which is disjoint from $H$, is a 
rational $(-4)$-curve $\Sigma$ in the homology class $[\Sigma]=[\wt{E}_0] - (\, [\wt{E}_1] + [\wt{E}_2] + [\wt{E}_3]\,)$. 

At this point we use the Nakai-Moishezon criterion and conclude that there exists a Kähler form $\wt{\omega}_4$ with the periods $\int_H \wt{\omega}_4=1$ and $\wt{\mu}_i=\int_{\wt{E}_i} \wt{\omega}_4$. Since $\Sigma$ is an $\wt{\omega}_4$-symplectic sphere of the area $4\epsi$, we make the rational blow-down of $\Sigma$ from $\wt{X}_4$ and obtain the manifold $X_3$ with the desired symplectic form $\omega_3$ on $X_3$ (with the prescribed periods and with an $\omega_3$-Lagrangian $L\cong \rp^2$ in $X_3$.) 

We will now give more details about applying the 
Nakai-Moishezon criterion in this particular situation.
We let $\calk(\wt{X}_4)$ to denote 
the Kähler cone of $\wt{X}_4$. 
\begin{lem} 
The cone $\calk(\wt{X}_4)$ consists 
of those classes 
\begin{equation} \label{w.in.X4} \textstyle
[\wt\omega_4]= \lambda[H] - \sum_{i=0}^3 \wt\mu_i [\wt{E}_i] \in \sfh^2(\wt{X}_4;\rr)
\end{equation}
which satisfy
\begin{enumerate}[label=$\wt{(\arabic*)}$]
\item $[\wt\omega_4]^2= \lambda^2 - \sum_{i=0}^3 \wt\mu_i^2>0$;
\item \label{some.E}   
$\wt\mu_i>0$ for $i = 0,\ldots,3$ and $\wt\mu_0 + \wt\mu_i < \lambda$ for $i=1,2,3$;
\item $\wt\mu_0 - ( \wt\mu_1 + \wt\mu_2 + \wt\mu_3)>0$. 
\end{enumerate}
\end{lem}
\begin{proof}
Let us first introduce more notations. 
The pencil of lines passing through the point $\wt{x}_0 \in \cp^2$ yields the holomorphic ruling 
$\pr_1 \colon \wt{X}_1 \to H$ for which 
$\wt{E}_0$ is section of self-intersection number $(-1)$.
The fibers of $\pr_1$ are in the class $[F] := [H] - [\wt{E}_0]$.

We let $\pr_4 \colon \wt{X}_4 \to H$ to denote 
the composition of the contractions of $\wt{E}_i, i = 1,2,3$ from $\wt{X}_4$ with the ruling $\pr_1$.
While the generic fiber of $\pr_4 \colon \wt{X}_4 \to H$ is a smooth holomorphic sphere in the class $[F]$, three fibers of $\pr_4$ are singular; each of them consists of two holomorphic exceptional curves, $\wt{E}_i, \wt{E}^{\prime}_{i}, i = 1,2,3$. 
The homology class of 
$\wt{E}^{\prime}_{i}$ is 
$[\wt{E}_{i}'] = [F] - [\wt{E}_i] = 
[H] - [\wt{E}_0] - [\wt{E}_{i}], i = 1,2,3$.   
\smallskip%

Going back to the proof of the lemma, note that it is sufficient to do 
the rational classes $\sfh^2(\wt{X}_4;\qq)$, as $\calk(\wt{X}_4)$ is an open convex cone, in which rational points are dense. Recall that a class $\xi \in \sfh^2(\wt{X}_4;\qq)$ has a Kähler representative if and 
only if $\xi^2 > 0$ and $\int_{C} \xi > 0$ for each (irreducible) holomorphic curve $C$. (Note that $\sfh^{1,1}(\wt{X}_4) = \sfh^2(\wt{X}_4;\cc)$, so that every integral class is the Chern class for some holomorphic line bundle.) Let us show that the classes $[\wt\omega_4]$ provided by the lemma are indeed positive on holomorphic curves. Consider the following cases:
\begin{itemize}
\item If $C$ is $\Sigma$, then the positivity 
follows from ($\wt{3}$).
    
\item If $[C] \cdot [F] = 0$, then $C$ is either a regular or a singular fiber of $\pr_4$, in which case the positivity follows from ($\wt{2}$).
    
\item In the last, the most general case, we have $[C] \cdot [F] > 0$ and $C \neq \Sigma$.
\end{itemize}
Set $d:= [C] \cdot [F]$ and $n'_i:= [C] \cdot [\wt{E}^{'}_{i}]$. 
Then $0 \leq n'_i \leq d$, as $[F] = [\wt{E}_i] + [\wt{E}'_{i}]$. 
Thus, we have:
\begin{equation}\label{Sigma1} \textstyle
[C] = d[\Sigma] + m[F] - \sum _{i=1}^3 n'_i [\wt{E}^{'}_{i}].
\end{equation}
Since $[C] \cdot [\Sigma] > 0$, 
it follows that 
$m - 4\,d > 0$. Therefore, one can rewrite \eqref{Sigma1} as follows:
\begin{equation*} \textstyle
[C] = d[\Sigma] + (m-3d) [F] +  \sum _{i=1}^3 (d-n'_i)[F] 
+\sum _{i=1}^3 n'_i\,([F] -[\wt{E}^{'}_{i}]).
\end{equation*}
Clearly, $[\wt\omega_4]$ is non-negative on each summand and positive $d[\Sigma]$. \qed
\end{proof}




\begin{thebibliography}{A-M-P}

\ifx\undefined\bysame
\newcommand{\bysame}{\leavevmode\hbox to3em{\hrulefill}\,}
\fi

\def\entry#1#2#3#4\par{\bibitem[#1]{#1}
{\textsc{#2 }}{\sl{#3} }#4\par\vskip2pt}

\def\noentry#1#2#3#4\par{}

\def\mathrev#1{{{\bf Math.\ Rev.:\,}{#1}}} 
\let\MR=\mathrev
\def\zbl#1{{{\bf ZBl.:\,}{#1}}} 

\def\arxiv#1{{{ArXiv: }{\tt{#1}}}} 


\entry{Aud}{Audin, M.:}%
{Quelques remarques sur les surfaces lagrangiennes de Givental$'$.} 
J.~Geom.\ Phys., {\bf7}(1990), 583--598,
\mathrev{1131914 (92i:57022)}

\noentry{AMP}%
{Auroux, D.; Mu{\~n}oz, V.; Presas, F.:}
{Lagrangian submanifolds and Lefschetz pencils.}
J.~Symplectic Geom.,  {\bf3}(2005),  171--219,
\mathrev{2199539}

\entry{BLW}{Borman, M.; Li, T.-J.; Wu, W.:}
{Spherical Lagrangians via ball packings and symplectic cutting.}
Selecta Mathematica, New Ser., {\bf20} (2014), No.\,1, 261--283,
\zbl{1285.53074}.


\noentry{Dol}
{Dolgachev, I.~V.} {Rational surfaces with a pencil of elliptic curves.} 
(Russian), Izv.\ Akad.\ Nauk SSSR, Ser.\ Mat., {\bf30} (1966), pp.1073--1100.  \mathrev{MR218356, 14.35},
\zbl{0187.18702}

\entry{F-S}{Fintushel, R.; Stern, R.:}
{Rational blowdowns of smooth 4-manifolds}
{Journal of Differential Geometry}, (1997)

\entry{Gm}{Gompf, R.:}
{A new construction of symplectic manifolds.}
{Annals of Mathematics}, (1995)

\entry{Gro}{Gromov, M.:}{Pseudo holomorphic curves in symplectic
manifolds.} Invent.\ Math. {\bf82} (1985), 307--347.

\entry{HO}%
{Hind, R.; Opshtein, E.:}%
{Squeezing Lagrangian tori in dimension $4$.}
Preprint, \arxiv{arXiv:1901.02124}.

\entry{Kh-1}%
{Khodorovsky, T.:}%
{Symplectic Rational Blow-up.}
Preprint, \arxiv{1303.2581v1}.

\entry{Kh-2}%
{Khodorovsky, T.:}%
{Bounds on embeddings of rational homology balls in symplectic 4-manifolds.}
Preprint, \arxiv{arXiv:1307.4321}.

\entry{LaMc}{Lalonde, F.; McDuff, D.:}{The classification of ruled symplectic manifolds.} Math.\ Res.~Lett., {\bf3}(1996), 769--778.

\entry{LL}{Li, T.-J.; Liu, Ai-Ko :}
{Uniqueness of Symplectic Canonical Class, Surface Cone and Symplectic Cone of 4-Manifolds with $b_+=1$.}
J.~Diff.\ Geom., {\bf58}(2001), n.2, 331--370.

\entry{MW}{McCarthy, J.; Wolfson, J.:}
{Symplectic normal connect sum.}
Topology, {\bf33}(1994), issue\,4, 729--764.

\entry{LW}{Li, T.-J.; Wu, W.:}
{Lagrangian spheres, symplectic surfaces and the symplectic mapping class group.}
Geom. Topol., {\bf16}(2012), n.2, 1121--1169.

\entry{N}{Nemirovski, S.:}
{Homology class of a Lagrangian Klein bottle.}
Izvestiya: Mathematics, (2009) 73:4, pp. 689–698.

\entry{Sym-1}{Symington, M.:}
{Symplectic rational blowdowns}
{Journal of Differential Geometry}, (1998)

\entry{Sym-2}{Symington, M.:}
{Four dimensions from two in symplectic topology},
Preprint, \arxiv{0210033v1}.

\entry{Sh}{Shevchishin, V.:}
{Lagrangian embeddings of the Klein bottle and combinatorial properties of mapping class groups.}
Izvestiya: Mathematics, (2009) 73:4, pp. 797–859.











\end{thebibliography}
\end{document}